\documentclass[11pt]{amsart}

\usepackage{hyperref}
\usepackage[capitalize]{cleveref}
\usepackage{amssymb}

\newcommand{\pref}[1]{\textup{(\ref{#1})}}
\newcommand{\fullref}[2]{\ref{#1}\pref{#1-#2}}
\newcommand{\fullcref}[2]{\cref{#1}\pref{#1-#2}}
\newcommand{\fullCref}[2]{\Cref{#1}\pref{#1-#2}}
\newcommand{\csee}[1]{(see \cref{#1})}

\numberwithin{equation}{section}

\newtheorem{cor}[equation]{Corollary}
\newtheorem{lem}[equation]{Lemma}
\newtheorem{prop}[equation]{Proposition}
\newtheorem{thm}[equation]{Theorem}

\newtheorem{copythm}{Theorem}
\numberwithin{copythm}{section}

\crefformat{cor}{Corollary~#2#1#3}
\crefformat{lem}{Lemma~#2#1#3}
\crefformat{prop}{Proposition~#2#1#3}
\crefformat{thm}{Theorem~#2#1#3}

\crefname{lem}{lemma}{lemmas}
\Crefname{lem}{Lemma}{Lemmas}
\crefname{prop}{proposition}{propositions}
\Crefname{prop}{Proposition}{Propositions}
\Crefname{thm}{Theorem}{Theorems}
\crefname{thm}{theorem}{theorems}

\crefmultiformat{thm}{Theorems~#2#1#3}{ and~#2#1#3}{, #2#1#3}{ and~#2#1#3}

\theoremstyle{definition}
\newtheorem*{ack}{Acknowledgements}
\newtheorem{assump}[equation]{Assumption}
\newtheorem{defn}[equation]{Definition}
\newtheorem{notation}[equation]{Notation}
\newtheorem{rem}[equation]{Remark}
\newtheorem{rems}[equation]{Remarks}

\crefformat{rems}{Remark~#2#1#3}

\crefformat{figure}{Figure~#2#1#3}

\newcounter{caseholder} 
\newcounter{case}
\numberwithin{case}{caseholder}
\newenvironment{case}[1][\unskip]{\refstepcounter{case}\bf\sffamily
\unskip\medskip \indent Case \thecase\ #1.\ \it}{\unskip\upshape}
\renewcommand{\thecase}{\arabic{case}}
\crefformat{case}{Case~#2#1#3}
\Crefname{case}{Case}{Cases}
\crefname{case}{case}{cases}

\newcounter{subcase}

\numberwithin{subcase}{case}
\crefformat{subcase}{Subcase~#2#1#3}
\crefname{subcase}{subcase}{subcases}
\Crefname{subcase}{Subcase}{Subcases}

\DeclareMathOperator{\Aut}{Aut}
\DeclareMathOperator{\Cay}{Cay}
\newcommand{\Cayd}{\mathop{\mathsf{C}\mkern-1mu\overrightarrow{\mathsf{ay}}}}

\newcommand{\ZZ}{\mathbb{Z}}

\newcommand{\dX}{\overrightarrow{X}}
\newcommand{\iso}{\cong}
\newcommand{\wG}{\widehat G}
\newcommand{\wS}{\widehat S}
\newcommand{\wX}{\widehat X}

\newcommand{\edge}{\mathrel{\vrule width 10pt height 3pt depth-2.5pt}} 
\renewcommand{\MR}[1]{\href{https://mathscinet.ams.org/mathscinet-getitem?mr=#1}{MR\,#1}}

\makeatletter
\newcommand{\noprelistbreak}{\smallskip\@nobreaktrue\nopagebreak} 
\makeatother

\begin{document}

\title[Non-Cayley-Isomorphic Cayley graphs]
{Non-Cayley-Isomorphic Cayley graphs \\ from Non-Cayley-Isomorphic Cayley digraphs}

\author{Dave Witte Morris and Joy Morris}
\address{%
\hskip-\parindent 
Department of Mathematics and Computer Science, University of Lethbridge, 
\newline 
4401 University Drive, Lethbridge, Alberta, T1K~3M4, Canada}
\email{dmorris@deductivepress.ca, https://deductivepress.ca/dmorris}
\email{joy.morris@uleth.ca}

\date{\today}

\begin{abstract}
A finite group~$G$ is a \emph{non-DCI group} if there exist subsets $S_1$ and~$S_2$ of~$G$, such that the associated Cayley digraphs $\Cayd(G;S_1)$ and $\Cayd(G;S_2)$ are isomorphic, but no automorphism of~$G$ carries $S_1$ to~$S_2$. Furthermore, $G$ is a \emph{non-CI group} if the subsets $S_1$ and~$S_2$ can be chosen to be closed under inverses, so we have undirected Cayley graphs $\Cay(G;S_1)$ and $\Cay(G;S_2)$.

We show that if $p$ is a prime number, and the elementary abelian $p$-group $(\ZZ_p)^r$ is a non-DCI group, then $(\ZZ_p)^{r+3}$ is a non-CI group. In most cases, we can also show that $(\ZZ_p)^{r+2}$ is a non-CI group. In particular, from Pablo Spiga's proof that $(\ZZ_p)^8$ is a non-DCI group, we conclude that $(\ZZ_3)^{10}$ is a non-CI group. This is the first example of a non-CI elementary abelian $3$-group.
\end{abstract}

\maketitle

\setcounter{section}{-1}

\section{Preliminaries}

We state some basic definitions, in order to establish our conventions.

\begin{defn}[cf.\ {\cite[pp.~302 and 307]{Li-survey}}] \label{BasicDefs}
Let $G$ be a finite group. 
	\begin{enumerate} \itemsep=\smallskipamount 
	\item For any subset~$S$ of~$G$, the \emph{Cayley digraph} $\Cayd(G; S)$ is the directed graph with vertex set~$G$, and with a directed edge $g \to h$ if and only if $h \in Sg$.
	 \item If $S$ is \emph{symmetric} (i.e., if $S = S^{-1}$), then $\Cayd(G; S)$ can be viewed as an undirected graph, by replacing each pair of oppositely directed edges ($x \to y$ and $y \to x$) with an undirected edge ($x \edge y$), and replacing each directed loop ($x \to x$) with an undirected loop ($x \edge x$). This undirected graph is called a \emph{Cayley graph}, and is denoted $\Cay(G;S)$. 
	\item We often refer to $S$ as the \emph{connection set} of $\Cayd(G; S)$ or $\Cay(G; S)$.
	\item A Cayley graph $\Cay(G; S)$ is said to be a \emph{CI graph} if, for every symmetric subset~$S'$ of~$G$, such that $\Cay(G; S) \iso \Cay(G; S')$, there is an automorphism~$\alpha$ of~$G$, such that $\alpha(S) = S'$. 
	\item Similarly, a Cayley digraph $\Cayd(G; S)$ is said to be a \emph{DCI digraph} if, for every subset~$S'$ of~$G$, such that $\Cayd(G; S) \iso \Cayd(G; S')$, there is an automorphism~$\alpha$ of~$G$, such that $\alpha(S) = S'$. 
	\item \label{BasicDefs-CI}
	$G$ is a \emph{CI group} if every Cayley graph of~$G$ is a CI graph, and
	\item \label{BasicDefs-DCI}
	$G$ is a \emph{DCI group} if every Cayley digraph of~$G$ is a DCI digraph.
	\end{enumerate}
Thus, the difference between ``CI'' and ``DCI'' is whether only undirected graphs are considered (so the generating set~$S$ is required to be symmetric), or all digraphs are allowed (so $S$ can be any subset of~$G$).
\end{defn}

\begin{rem}
The terminology in~\pref{BasicDefs-CI} and~\pref{BasicDefs-DCI} is not entirely standard: some authors use ``CI'' for the concept that we call ``DCI\rlap,'' and may use a phrase such as ``CI with respect to undirected graphs'' for the concept that we call ``CI\rlap.'' The letters ``CI'' are an abbreviation of ``Cayley isomorphism'' \cite[\S3]{Li-survey}. 
\end{rem}

\section{Introduction}

Most finite groups are not CI groups. In particular, it is known that every CI group has a subgroup of index at most 24 that is a direct product of elementary abelian groups \cite[Thm.~1.2]{LiLuPalfy}. (Recall that a group is \emph{elementary abelian} if it is isomorphic to $(\ZZ_p)^r$, for some prime number~$p$, and some $r \ge 0$.)
However, it is not known which groups with this property are indeed CI groups. (See the survey \cite{Li-survey}.) Therefore, a fundamental problem in the theory of CI groups is to determine which elementary abelian groups are CI groups. Since every subgroup of a CI group is a CI group \cite[Lem.~3.2]{BabaiFrankl-I}, it suffices to determine, for each prime~$p$, the smallest natural number~$r$, such that 
$(\ZZ_p)^r$ is \emph{not} a CI group. This number is~$6$ for $p = 2$ \cite{ConderLi,Nowitz}, and it is known that the number is at least~$6$ for all~$p$ \cite{FengKovacs}, but the exact value is unknown for every $p > 2$. 

To provide an upper bound,  G.\,Somlai \cite[Thm.~2]{Somlai} showed in 2011 that if $p > 3$, then the elementary abelian group of order~$p^{2p+3}$ is not a CI group. 
However, Somlai \cite[p.~324]{Somlai} also pointed out that the case $p = 3$ remained open: it was not known whether there is an elementary abelian $3$-group that is not a CI group, even though P.\,Spiga \cite[Thm.~2]{Spiga-3grps} had proved two years earlier that the elementary abelian group of order~$3^8$ is not a DCI group \csee{SpigaEg}, and larger examples of non-DCI elementary abelian $3$-groups had been constructed previously \cite{Muzychuk-ElemAbel,Spiga-pgrps}.  To fill this gap in the literature, we show that every example of a non-DCI elementary abelian $p$-group will automatically yield an example of a (slightly larger) non-CI elementary abelian $p$-group (for the same prime~$p$), even if $p = 3$:

\begin{thm} \label{r+3}
If $(\ZZ_p)^r$ is not a DCI group \textup(and $p$ is prime\textup), then $(\ZZ_p)^{r+3}$ is not a CI group.
\end{thm}

Moreover, the following result shows that we can usually decrease the exponent in our answer by~$1$. (In fact, it will be explained in \fullcref{IntroRems}{r+2} that that the exponent $r + 2$ will always suffice, but this improvement relies on a paper that is still in preparation.)

\begin{thm} \label{r+2}
Assume $(\ZZ_p)^r$ is not a DCI group \textup(and $p$ is prime\textup). If either of the following is true, then $(\ZZ_p)^{r+2}$ is not a CI group:
\noprelistbreak
	\begin{enumerate}
	\item \label{r+2-not3}
	$p \neq 3$,
	or
	\item \label{r+2-3}
	there is a non-DCI Cayley digraph $\Cayd \bigl( (\ZZ_p)^r; S \bigr)$, such that $|S| < p^{r-1}$.
	\end{enumerate}
\end{thm}

\begin{figure}[t] 
\begin{align*}
S_{0,0,0} &= \{v_1, v_3, v_4, v_5\}\\
S_{1,0,0} &= \{w_1+av_1+bv_2+cv_5 : a,b,c \in \mathbb Z_3\}\\
S_{0,1,0} &= \{w_2+av_1+bv_3+cv_4+dv_5 : a,b,c,d \in \mathbb Z_3\}\\
S_{0,0,1} &= \{w_3+av_2+bv_3+cv_4+dv_5 : a,b,c,d \in \mathbb Z_3\}\\
S_{1,1,0} &= \{w_1+w_2+av_1+bv_2+cv_3+bv_4+dv_5 : a,b,c,d \in \mathbb Z_3\}\\
S_{1,0,1} &= \{w_1+w_3+av_1+bv_2+av_3+cv_4+dv_5 : a,b,c,d \in \mathbb Z_3\}\\
S_{0,1,1} &= \{w_2+w_3+av_1+bv_2+cv_3+dv_4-(a+b)v_5 : a,b,c,d \in \mathbb Z_3\}\\
S_{1,1,1} &= \{w_1+w_2+w_3+av_1+bv_2+cv_3+dv_4+(-a-b+c+d)v_5 : a,b,c,d \in \mathbb Z_3\}\\
S_{2,1,1} &= \{2w_1+w_2+w_3+av_1+bv_2+cv_3+dv_4-(a+b+c+d)v_5 : a,b,c,d \in \mathbb Z_3\}\\
S_{1,2,1} &= \{w_1+2w_2+w_3+av_1+bv_2+cv_3+dv_4+(a+b-c+d)v_5 : a,b,c,d \in \mathbb Z_3\}\\
S_{1,1,2} &= \{w_1+w_2+2w_3+av_1+bv_2+cv_3+dv_4+(a+b+c-d)v_5 : a,b,c,d \in \mathbb Z_3\}
\end{align*}
\caption{Eleven sets $S_{i,j,k}$ whose union is the connection set for P.\,Spiga's non-DCI Cayley digraph of $(\ZZ_3)^8$.}
\label{SpigaSets}
\end{figure}

For the reader's convenience, we recall the detailed statement of the following important example (that was mentioned above).

\begin{thm}[{Spiga \cite[proof of Thm.~2]{Spiga-3grps}}] \label{SpigaEg}
Let $\{w_1, w_2, w_3, v_1, v_2, v_3, v_4, v_5\}$ be a generating set of~$(\ZZ_3)^8$, and let $S$ be the union of the  sets $S_{i,j,k}$ that are defined in \cref{SpigaSets}.
Then $\Cayd \bigl( (\ZZ_3)^8 ; S \bigr)$ is a non-DCI digraph.
\end{thm}

 We see from \cref{SpigaSets} that the outvalency of this Cayley digraph (i.e., the cardinality of the connection set~$S$) is $4 + 3^3 + 9\cdot 3^4 = 760$. Since $760 < 3^7$ ($= 2187$), we immediately deduce the following from \fullcref{r+2}{3}:

\begin{cor} \label{3^10}
$(\ZZ_3)^{10}$ is not a CI group.
\end{cor}

\begin{rems} \label{IntroRems}
\leavevmode\noprelistbreak
	\begin{enumerate} \itemsep=\smallskipamount
	\item By modifying Spiga's example, the second author \cite{Morris-3^8} has shown that $(\ZZ_3)^8$ is not a CI group. Therefore, \cref{3^10} may be of limited interest.  However, the argument in \cite{Morris-3^8} is  longer (and more intricate) than the proofs here. (Part of the reason our proof of \cref{3^10} is short is that we assume it is known that Spiga's example is not a DCI digraph, but \cite{Morris-3^8} also makes this assumption.)

	\item \label{IntroRems-r+2}
	If we assume the above-mentioned fact that $(\ZZ_3)^8$ is not a CI group, then the conclusion of \cref{r+2} is true even without assuming \pref{r+2-not3} or~\pref{r+2-3}: 
	\[ \text{\emph{if $(\ZZ_p)^r$ is not a DCI group, then $(\ZZ_p)^{r+2}$ is not a CI group.}} \]
To see this, note that we may assume $p = 3$ (for otherwise \fullcref{r+2}{not3} applies). However, P.\,Spiga \cite[Thm.~1]{Spiga-3grps} showed that $(\ZZ_3)^5$ is a DCI group, so we must have $r \ge 6$. Then $r + 2 \ge 8$. If we assume that $(\ZZ_3)^8$ is not a CI group, then this implies $(\ZZ_p)^{r+2}$ is not a CI group, as desired. Note that part~\pref{r+2-3} of \cref{r+2} was not used in this argument. In fact, knowing that $(\ZZ_3)^8$ is not a CI group makes this part of the \lcnamecref{r+2} superfluous.

	\item \fullCref{r+2}{not3} may have a better chance of being useful in the future. For example, if it is ever proved that there is a constant~$C$, such that, for every prime~$p$, the elementary abelian group of order~$p^C$ is not a DCI group, then it will immediately follow that there is a constant~$C'$, such that the elementary abelian group of order~$p^{C'}$ is not a CI group.
	\end{enumerate}
\end{rems}

Our construction of a non-CI Cayley graph $\Cay(\wG; \wS)$ from a non-DCI Cayley digraph $\Cayd(G; S)$ is in \cref{XhatDefn}. It is a fairly straightforward adaptation of the well-known observation that if $\widetilde X$ is the bipartite double cover of a digraph~$\dX$, i.e., if:
	\begin{itemize}
	\item $V(\widetilde X) = V(\dX) \times \{0,1\}$, 
	and
	\item $(x,0)$ is adjacent to $(y,1)$ in~$\widetilde X$ $\iff$ there is a directed edge from $x$ to~$y$ in~$\dX$,
	\end{itemize}
and we define a permutation $\widetilde\pi$ of $V(\widetilde X)$ by $\widetilde\pi(x,i) = \bigl( \pi(x), i \bigr)$, where $\pi$ is a  permutation of $V(\dX)$, then
	\[ \text{$\widetilde\pi$ is an automorphism of the graph~$\widetilde X$ 
		$\iff$ $\pi$ is an automorphism of the directed graph~$\dX$} . \]

\begin{ack}
The second author was partially supported by the Natural Science and Engineering Research Council of Canada (grant RGPIN-2017-04905).
\end{ack}

\section{Proofs of the main results} \label{MainPfSect}

This \lcnamecref{MainPfSect} proves \cref{MainThm}, which easily implies the main results that were stated in the Introduction 
(i.e.,  \cref{r+3,r+2}).
Although elementary abelian groups are our primary interest, the result is stated more generally, because assuming that the group~$G$ is (elementary) abelian would not simplify the argument to any significant extent.

\begin{notation}
If $\dX$ is a digraph, then $\dX^-$ is the digraph that is obtained from~$\dX$ by reversing all of its directed edges.
\end{notation}

\vbox{
\begin{thm} \label{MainThm}
Assume 
	\begin{itemize}
	\item $G$ is a finite group,
	\item $\dX = \Cayd(G; S)$ is a non-DCI Cayley digraph,
	\item $n \ge 3$
	and
	$k = \begin{cases} 2 & \text{if $n > 3$}; \\ 3 & \text{if $n = 3$} , \end{cases}$
	\item $nk \neq |G|$,
	\item either $n > 3$ or $|S \cup \{1_G\}| \le |G|/3$,
	and
	\item either $\dX \not\cong \dX^-$ or there is an automorphism~$\alpha$ of~$G$, such that $\alpha(S) = S^{-1}$.
	\end{itemize}
Then $G \times (\ZZ_n)^2$ is not a CI group.
\end{thm}}
\setcounter{copythm}{\arabic{equation}} 

Before proving this \lcnamecref{MainThm}, let us show that it contains \cref{r+2} as a special case, and then derive \cref{r+3} from this latter result.

\begin{proof}[\bf Proof of \cref{r+2}]
For $p = 2$, it is well known that $(\ZZ_2)^r$ is a CI group if and only if it is a DCI group. (Every element of $(\ZZ_2)^r$ is equal to its inverse, so every generating set is symmetric. Hence, every Cayley digraph of $(\ZZ_2)^r$ is also a Cayley graph.) Therefore, we may assume $p > 2$.

Let $G = (\ZZ_p)^r$ and $n = p$. Since $(\ZZ_p)^r$ is not a DCI group, there is a non-DCI Cayley digraph $\dX = \Cayd(G; S)$.
Also note that, since $G$ is abelian, the function $\alpha(g) = g^{-1}$ is a group automorphism, and it obviously has the property that $\alpha(S) = S^{-1}$.

Assume, for the moment, that $p > 3$. Then $k = 2$ is not a power of~$p$, so it is obvious that $nk \neq p^r = |G|$. Therefore, all of the hypotheses of \cref{MainThm} are satisfied, so $G \times (\ZZ_p)^2$ is not a CI group.

We now assume $p = 3$. In this case, Condition~\fullref{r+2}{3} must apply, so we may assume $|S| \le 3^{r-1} - 1$. Then
	\[ |S \cup \{1_G\}| \le |S| + 1 \le (3^{r-1} - 1) + 1 = 3^{r-1} = |G|/3 . \]
Also, it is well known that $(\ZZ_p)^2$ is a DCI group (in fact, even $(\ZZ_p)^5$ is a DCI group \cite{FengKovacs}), so $|G| \neq 3^2 = 3 \cdot 3 = nk$. 
Therefore, all of the hypotheses of \cref{MainThm} are satisfied again, so $G \times (\ZZ_p)^2$ is not a CI group.
\end{proof}

\begin{proof}[\bf Proof of \cref{r+3}]
We may assume $p = 3$, for otherwise \fullcref{r+2}{not3} applies. Then, by assumption, there is a non-DCI, Cayley digraph $\Cayd \bigl( (\ZZ_3)^r; S)$. Via the natural embedding of $(\ZZ_3)^r$ in~$(\ZZ_3)^{r+1}$, we also have the Cayley digraph $\Cayd \bigl( (\ZZ_3)^{r+1}; S)$, which is also non-DCI. Since $S \cup \{1\} \subseteq (\ZZ_3)^r$, we have 
	$|S \cup \{1\}| \le |(\ZZ_3)^r| = |(\ZZ_3)^{r+1}| /3$,
so we conclude from \fullcref{r+2}{3} (with $r + 1$ in the place of~$r$) that $(\ZZ_3)^{(r+1) + 2}$ is not a CI group. 
\end{proof}

The remainder of this \lcnamecref{MainPfSect} will prove \cref{MainThm}. To get started, we fix some notation.

\begin{notation} \label{XhatDefn}
Let 
	\begin{itemize}
	\item $\dX = \Cayd(G; S)$ be a Cayley digraph of a nontrivial finite group~$G$, such that $S \neq S^{-1}$ (so $\dX$ is not an undirected graph),
	\item $n \ge 3$
		and
		 $k = \begin{cases} 2 & \text{if $n > 3$}; \\ 3 & \text{if $n = 3$} , \end{cases}$
	\item $A$ be a group of order~$n$,
	\item $B = \langle b \rangle$ be a \textup(multiplicative\textup) cyclic group of order~$n$,
	\item $\wG = G \times A \times B$ (so $G$, $A$, and~$B$ can be naturally identified with subgroups of~$\wG$ that centralize each other),
	\item $\wS = G \cup Sb \cup A \cup Ab$,
	and
	\item $\wX = \Cay \bigl( \wG, \wS^{\pm 1} \bigr)$.
	\end{itemize}
Note that, since $G$, $A$, and~$B$ are normal subgroups of~$\wG$, there is no need to specify whether cosets of these subgroups are right cosets or left cosets.
\end{notation}

\begin{rem} 
In our applications, we will take $G = (\ZZ_p)^r$ and $A = B = \ZZ_p$, for some prime~$p$. To prove \cref{MainThm}, we will show that if $\dX$ is not a DCI digraph, and some minor technical conditions are satisfied, then $\wX$ is not a CI graph.
\end{rem}

\begin{notation}[{\cite[\S2]{Babai-IsoProb}}]
For any group~$H$, we use $H_R$ to denote the \emph{right regular representation} of~$H$. This consists of all permutations of~$H$ that have the form $x \mapsto xh$, for some $h \in H$.
\end{notation}

The following well known, fundamental \lcnamecref{Babai} shows that being a CI graph (or CI digraph) is a property of the automorphism group of the graph (or digraph).

\begin{thm}[Babai {\cite[Lem.~3.1]{Babai-IsoProb}}, or {\cite[Thm.~4.1]{Li-survey}}] \label{Babai}
A Cayley graph \textup(or Cayley digraph\textup)~$Y$ of a group~$H$ is a CI graph \textup(or DCI digraph\textup) if and only if the conjugates of~$H_R$ are the only subgroups of $\Aut Y$ that are isomorphic to~$H$ and act sharply transitively on $V(Y)$.
\end{thm}

Adding a directed loop at every vertex of a digraph does not affect the automorphism group of the digraph. Therefore, the following causes no loss of generality:

\begin{assump} \label{1inS}
Assume $1_G \in S$, which means that $\Cayd(G; S)$ has a directed loop at every vertex.
\end{assump}

\begin{rem}
Since $b \in Ab$, we know that $b \in \wS$, even without \cref{1inS}. Therefore, adding $1_G$ to~$S$ does not change $\wS$ at all. The reason for making \cref{1inS} is to ensure that $Sb$ is the set of all outneighbours of~$1_{\wG}$ that are in the coset $Gb$. This avoids needing to treat the vertex~$b$ as a special case when looking at the outneighbours of~$1_G$.
\end{rem}

The following simple result provides a crucial connection between $\Aut \dX$ and $\Aut \wX$. 

\begin{prop} \label{IfFixG}
Suppose $\varphi$ is an automorphism of~$\wX$, such that
	\begin{enumerate}
	\item  \label{IfFixG-1}
	$\varphi(1) = 1$,
	\item \label{IfFixG-cosetG}
	$\varphi$ maps each coset of $G$ to a coset of~$G$,
	and
	\item \label{IfFixG-cosetAB}
	$\varphi$ maps each coset of $AB$ to a coset of~$AB$.
	\end{enumerate}
Then the restriction of~$\varphi$ to~$G$ is either 
	an automorphism of~$\dX$ or an isomorphism from $\dX$ to $\dX^-$.
\end{prop}

\begin{proof}
By~\pref{IfFixG-cosetG}, we know that $\varphi$ maps $G$ to some coset of~$G$. Since $\varphi(1) = 1$, this implies that 
	\[ \varphi(G) = G . \]
So the restriction of $\varphi$ to~$G$ is a permutation of~$G$.

From the definition of~$\wS$, we see that $1$ has $|S|$ neighbours in the coset $Gb$ and also has $|S|$ neighbours in the coset~$Gb^{-1}$ (and has $|G|$ neighbours in its own coset~$G$) but has only one neighbour in any other coset of~$G$. Therefore, 
	\begin{align*}
	\text{$\varphi(Gb) = Gb^\epsilon$, \ for some $\epsilon \in \{\pm1\}$.}
	\end{align*}
Let $\dX^\epsilon$ be $\dX$ or~$\dX^-$, depending on whether $\epsilon$ is~$1$ or~$-1$, respectively.

\medbreak

We claim that $\varphi(gb) = \varphi(g) b^\epsilon$, for every $g \in G$. To see this, note that \pref{IfFixG-cosetAB} implies $\varphi(gb) \in \varphi(g) AB$. Also, we see from the above displayed equations that $\varphi(gb) \in Gb^\epsilon = \varphi(g) Gb^\epsilon$. Since $(AB) \cap G = \{1\}$, this implies the claim.

\medbreak

We can now complete the proof, by showing that the restriction of~$\varphi$ to~$G$ is an isomorphism from~$\dX$ to~$\dX^\epsilon$.
To this end, let $g,h \in G$, such that $g \to h$ in~$\dX$. This means there exists $s \in S$, such that $h = sg$. Since $sb \in Sb$, we have $g \edge (sb)g$ in~$\wX$. Since $\varphi \in \Aut \wX$, this implies
	\[ \varphi(g) \edge \varphi \bigl( (sb)g \bigr) = \varphi \bigl( (sg) b \bigr) = \varphi(s g) b^\epsilon . \]
Since $\varphi(g), \varphi(sg) \in G$, we see from the definition of~$\wS$ that this implies $\varphi(sg) = t^\epsilon \varphi(g)$, for some $t \in S$. Hence, there is a directed edge from $\varphi(g)$ to $\varphi(sg)$ in~$\dX^\epsilon$. 

We have shown that if $g \to h$ in~$\dX$, then $\varphi(g) \to \varphi(h)$ in~$\dX^\epsilon$. Since $\dX$ and $\dX^\epsilon$ are regular digraphs of the same outvalency, this implies that the permutation $\varphi|_G$ is an isomorphism of digraphs.
\end{proof}

Our next series of lemmas 
culminates in \cref{phiIsGood}, which establishes simple conditions that imply the hypotheses of the above 
\lcnamecref{IfFixG} are satisfied. That will complete our preparations for the proof of \cref{MainThm}.

\begin{lem} \label{cliques}
If $n > 3$, then every maximal clique of~$\wX$ is induced by one of the following sets of vertices \textup(for some $x \in \wG$\textup{):}
	\begin{enumerate} \renewcommand{\theenumi}{\alph{enumi}}
	\item \label{cliques-G}
	$G x$ \textup(i.e., a coset of the subgroup~$G$\textup),
	\item \label{cliques-A} 
	$Ax \cup Abx$,
	or
	\item \label{cliques-other}
	a subset of $G x \cap Gbx$ that does not contain any coset of~$G$.
	\end{enumerate}
If $n = 3$, then \pref{cliques-A} and~\pref{cliques-other} are replaced with:
	\begin{enumerate} 
	\renewcommand{\theenumi}{\ref{cliques-A}$'$}
	\item \label{cliques-A3}
	$ABx$, 
	and
	\renewcommand{\theenumi}{\ref{cliques-other}$'$}
	\item \label{cliques-other3}
	a subset of $GB x$ that does not contain any coset of~$G$.
	\end{enumerate}
Conversely, the subgraph induced by each subset listed in \pref{cliques-G} or~\textup{(\ref{cliques-A}/\ref{cliques-A3})} is a maximal clique.
\end{lem}

\begin{proof}
We first prove the ``converse'' that is stated the final sentence of the \lcnamecref{cliques}. Assume, without loss of generality, that $x = 1$.

\pref{cliques-G} Since $G$ is contained in~$\wS$, it is clear that the subgraph induced by $G$ is a clique. So we just need to show that the clique is maximal.  Suppose $y \in \wG$, such that $y \notin G$, but $y$ is adjacent to every vertex in~$G$. After translating by an element of~$G$, we may assume $y \in AB$. Since $y$ is adjacent to~$1$, we also know that $y \in \wS^{\pm1}$. Hence, $y \in A \cup Ab \cup Ab^{-1}$ (and $y \neq 1$). Since $S \neq S^{-1}$ \csee{XhatDefn}, there is some nontrivial $g \in G$, such that $g^{-1} \notin S$. Now, we see from the definition of~$\wS$ that:
	\begin{itemize}
	\item If $s \in \wS$, such that $sg \in AB$, and $sg \neq 1$, then $s \in (Sb)^{-1}$. Hence, $g$ is not adjacent to any element of $A$ or~$Ab$. Therefore, $y \notin A \cup Ab$. 
	\item Similarly, if $s \in \wS$, such that $s g^{-1} \in AB$, and $s g^{-1} \neq 1$, then $s \in Sb$. Hence, $g^{-1}$ is not adjacent to any element of~$Ab^{-1}$. Therefore, $y \notin Ab^{-1}$. 
	\end{itemize}
This is a contradiction. So we conclude that the clique induced by $G$ is indeed maximal.

\pref{cliques-A} Since $A$ is contained in~$\wS$, it is clear that $A$ and $Ab$ each induce a clique. Also, we know that every vertex in $A$ is adjacent to every vertex in $Ab$, because $Ab$ is contained in~$\wS$. Therefore, the subgraph induced by $A \cup Ab$ is a clique. 
So we just need to show that this clique is maximal.  Suppose $y \in \wG$, such that $y \notin A \cup Ab$, but $y$ is adjacent to every vertex in $A \cup Ab$. Since $y$ is adjacent to~$1$, we know that $y \in \wS^{\pm1}$. Since $y \notin A \cup Ab$, we conclude that $y \in G \cup Sb \cup Sb^{-1} \cup Ab^{-1}$. Since $n > 3$, no element of $GAb^{-1}$ is adjacent to any element of~$Ab$. Therefore, we must have $y \in G \cup Sb$ (and $y \neq 1,b$). Then every neighbour of~$y$ that is in $AB$ must be in 
	\[ Gy \cup Sby \cup (Sb)^{-1} y \subseteq GBy . \]
Since $GBy$ does not include any nontrivial elements of~$A$, this contradicts the fact that $y$ is adjacent to every vertex in~$A$. 

\pref{cliques-A3} 
Since $n = 3$, we have $AB \subseteq \wS$, so the subgraph induced by $AB$ is a clique. So we just need to show that this clique is maximal.  Suppose $y \in \wG$, such that $y \notin AB$, but $y$ is adjacent to every vertex in $AB$. We may assume $y \in G$ (after translating by an element of~$AB$). This implies that $y$ is not adjacent to any nontrivial element of~$A$  (see the proof of case~\pref{cliques-A}). This is a contradiction.

\medbreak

Now, let $C$ be any maximal clique in~$\wX$, and assume, without loss of generality, that $1 \in C$. This implies $C \subseteq \wS^{\pm1}$. We will prove that $C$ is contained in a clique of type \pref{cliques-G}, \pref{cliques-A}, \pref{cliques-A3}, \pref{cliques-other}, or~\pref{cliques-other3}.

\refstepcounter{caseholder}

\begin{case}
Assume $C$ is not contained in any coset of $GB$.
\end{case}
Then $C$ must contain some vertex of the form $gab^i$ where $g \in G$, $i \in \ZZ$, and $a$ is a nontrivial element of~$A$. Since $gab^i \in C \subseteq \wS^{\pm1}$, we must have $g = 1$ and $i \in \{0, \pm1\}$.

Now, let $v$ be any common neighbour of $1$ and $ab^i$. There is some $c \in \{1, ab^i\}$ that is not in $GBv$. Then (by the preceding argument) we have $v = a b^j c $, for some $j \in \{0, \pm1\}$. Hence, we have $v \in AB$. This shows that $C$ is contained in a maximal clique of type~\pref{cliques-A3} if $n = 3$.

So we may assume $n > 3$. Then no vertex in $Ab$ is adjacent to any vertex in~$Ab^{-1}$. Thus, $C$ must be contained in either $A \cup Ab$ or $Ab^{-1} \cup A$. Each of these sets is of type~\pref{cliques-A} (with $x \in \{1, b^{-1}\}$).

\begin{case}
Assume $C$ is contained in a coset of $GB$, and is not of type~\pref{cliques-G}.
\end{case}
Since $1 \in C$, we must have $C \subseteq GB$. We also know that $C$ does not contain any coset of~$G$ (since cosets of~$G$ induce maximal cliques of type~\pref{cliques-G}), so we may assume $n > 3$, for otherwise $C$ is of type~\pref{cliques-other3}. Note that 
	\[ C \subseteq GB \cap \wS^{\pm1} =  G \cup Gb \cup Gb^{-1} . \]
However, no vertex in $Gb$ is adjacent to any vertex in~$Gb^{-1}$ (since $n > 3$), so we conclude that $C$ is contained in either $G \cup Gb$ or $Gb^{-1} \cup G$. So $C$ is of type~\pref{cliques-other} (with $x \in \{1, b^{-1}\}$).
\end{proof}

The following observation is a variation of the easy (and well known) fact that if $x$ and~$y$ are two vertices of $\Cayd(G; S)$ that have the same outneighbours, then $S$ is a union of left cosets of the subgroup generated by $x y^{-1}$. 

\begin{lem} \label{AlmClosedIsSubgrp}
Assume $D$ is a set of vertices of $\Cayd(G; S)$, such that $|D| \ge \max \bigl( |S| - 1, 2 \bigr)$, and all of the vertices in~$D$ have exactly the same outneighbours. 
Then $S$ is a subgroup of~$G$.
\end{lem}

\begin{proof}
Assume, without loss of generality, that $1 \in D$. Then $S$ is the set of outneighbours of every element of~$D$, so $Sd = S$ for every $d \in D$. This means that $S$ is a union of left cosets of $\langle D \rangle$. If $S \neq  \langle D \rangle$, then, since $|S| \le |D| + 1$, this implies that the elements of~$D$ form one coset of $\langle D \rangle$, and the remaining element of~$S$ is an entire coset of $\langle D \rangle$. So $|\langle D \rangle| = 1$. This contradicts the assumption that $|D| \ge 2$.
\end{proof}

\begin{lem} \label{phiIsGood}
Assume 
\noprelistbreak
	\begin{itemize} \itemsep=\smallskipamount 
	 \item $\varphi$ is an automorphism of~$\wX$ that fixes the vertex $1_{\wG}$,
	 \item $nk \neq |G|$  \textup(recall that $n$ and~$k$ were defined in \cref{XhatDefn}\textup),
	 and
	 \item $|S| \le \bigl( |G| + 1 \bigr)/k$.
	 \end{itemize}
Then the hypotheses of \cref{IfFixG} are satisfied.
\end{lem}

\begin{proof}
\fullref{IfFixG}{1}: \emph{We have $\varphi(1) = 1$.}
This is true by assumption.

\fullref{IfFixG}{cosetG}: \emph{$\varphi$ maps each coset of $G$ to a coset of~$G$.}
The maximal cliques of~$\wX$ are described in \cref{cliques}. It is obvious that every clique of type~\pref{cliques-G} has cardinality~$|G|$, and that every clique of type~\pref{cliques-A} or~\pref{cliques-A3} has cardinality~$nk$ (since $|A| = n$). 
Assume, for the moment, that 
	\begin{align} \label{phiIsGoodPf-card}
	\begin{matrix}
	\text{\emph{the cardinality of every maximal clique of type~\pref{cliques-other}}} \\
	\text{\emph{\textup(or of type~\pref{cliques-other3} if $n = 3$\textup) is strictly less than~$|G|$.}} 
	\end{matrix}
	\end{align}
 (We will show how to complete the proof with this assumption, and we will establish later that the assumption is indeed true.) 
This assumption implies that the cosets of~$G$ are the only maximal cliques whose cardinality is~$|G|$. So every automorphism of~$\wX$ (including~$\varphi$) must map each coset of~$G$ to a coset of~$G$. 

\fullref{IfFixG}{cosetAB}: \emph{$\varphi$ maps each coset of $AB$ to a coset of~$AB$.}
Each maximal clique of type~\pref{cliques-A} or~\pref{cliques-A3} contains no more than one vertex of any coset of~$G$. If a maximal clique of type~\pref{cliques-other} or~\pref{cliques-other3} has this property, then it cannot have more than $k$ vertices. Since maximal cliques of type~\pref{cliques-A} or~\pref{cliques-A3} have $kn$ vertices (and $kn > k$), this implies that no automorphism can carry a maximal clique of type~\pref{cliques-A} or~\pref{cliques-A3} to a clique of type~\pref{cliques-other} or~\pref{cliques-other3}. So $\varphi$ must preserve the set of cliques of type~\pref{cliques-A} (or~\pref{cliques-A3}). This easily implies Hypothesis~\fullref{IfFixG}{cosetAB}.

\medbreak 

Now, all that remains is to prove Assumption~\pref{phiIsGoodPf-card}.
To this end, let $C$ be a maximal clique of type~\pref{cliques-other} (or of type~\pref{cliques-other3} if $n = 3$), and suppose $|C| \ge |G|$. (This will lead to a contradiction, which completes the proof.)
Also assume, without loss of generality, that $1_{\wG} \in C$. Then 
	\[ C \subseteq GB \cap \wS^{\pm1} = G \cup Sb \cup S b^{-1} \subseteq G \cup Gb \cup G b^{-1} . \]
Let $C_i = C \cap Gb^i$ for $i \in \{-1,0,1\}$. It is clear that $|C_1| \le |Sb| = |S|$ and $|C_{-1}| \le |Sb^{-1}| = |S|$.

We claim that also $|C_0| \le |S|$. Since $C \neq G$ and $|C| \ge |G|$, we know that $C \nsubseteq G$. Hence, $C_{-1} \cup C_1 \neq \emptyset$, so there is some $g b^\epsilon$ in~$C$, with $\epsilon \in \{\pm1\}$. The set of neighbours of this vertex in $G$ is $(Sb)^{-\epsilon} g b^\epsilon$. Since this set of neighbours contains~$C_0$ and has cardinality~$|S|$, we conclude that $|C_0| \le |S|$.

Choose $j \in \{-1,0,1\}$, such that $|C_j|$ is maximal. 
Since at most~$k$ of the sets $C_{-1}, C_0, C_1$ are nonempty, and 
	\[ |C_i| \le |C_j| \le |S| \le \bigl( |G| + 1 \bigr)/k, \] 
then we have
	\[ |G| \le |C| = |C_{-1}| + |C_0| + |C_1| \le k \cdot |C_j| \le k \cdot |S| \le k \cdot \bigl( |G| + 1 \bigr)/k = |G| + 1 ,\]
so
	\[ \frac{|G|}{k} \le |C_j| \le |S| \le \frac{|G|}{k} + \frac{1}{k} .\]
Since $1/k < 1$, this implies $|C_j| = |S|$ (because the absolute value of the difference of two distinct integers cannot be less than one). Also, since 
	\[ |C_{-1}| + |C_0| + |C_1| = |C| \ge |G| \ge k \cdot |C_j|  - 1 , \]
we have $|C_i|  \ge |C_j| - 1$ for all $i \neq j$.  Assume, for the sake of concreteness, that $j = 1$, so we may let $i = 0$.

Since $C$ is a clique, we know that every element of~$C_1$ is an outneighbour of every element of~$C_0$. From the definition of~$\wS$, we also know, for each $x \in C_0$, that the number of outneighbours of~$x$ in $Gb$ is precisely $|S| = |C_1|$. Hence, we conclude that $C_1$ is the set of all outneighbours of~$x$ that are in~$Gb$; so $C_1 = Sbx$. Since $x$ is an arbitrary element of~$C_0$ (and $1_G \in C_0$), we conclude that $Sb x = Sb \, 1_G$. Since $Sb = bS$ (recall that $b$ commutes with all elements of~$G$), this implies $Sx = S$, for all $x \in C_0$.

Now \cref{AlmClosedIsSubgrp} tells us that $S$ is a subgroup of~$G$. Therefore $S$ is closed under inverses (i.e., $S = S^{-1}$). This contradicts \cref{XhatDefn}.
\end{proof}

We are now ready to prove \cref{MainThm}. For the reader's convenience, we copy the statement here.

\addtocounter{copythm}{-1} 
\begin{copythm} 
Assume 
	\begin{itemize}
	\item $G$ is a finite group,
	\item $\dX = \Cayd(G; S)$ is a non-DCI Cayley digraph,
	\item $n \ge 3$
	and
	$k = \begin{cases} 2 & \text{if $n > 3$}; \\ 3 & \text{if $n = 3$} , \end{cases}$
	\item $nk \neq |G|$,
	\item either $n > 3$ or $|S \cup \{1_G\}| \le |G|/3$,
	and
	\item either $\dX \not\cong \dX^-$ or there is an automorphism~$\alpha$ of~$G$, such that $\alpha(S) = S^{-1}$.
	\end{itemize}
Then $G \times (\ZZ_n)^2$ is not a CI group.
\end{copythm}

\begin{proof}
We may assume $S \neq S^{-1}$, for otherwise $\Cay(G; S)$ is a non-CI Cayley graph of~$G$, so $G$ is not a CI group; hence any finite group that contains $G$ is also not a CI group.

Let $A = B = \ZZ_n$, and
recall that $\wG$, $\wS$, and $\wX = \Cay(\wG; \wS^{\pm1})$ are defined in \Cref{XhatDefn}. We will show that the Cayley graph~$\wX$ is not a CI graph.

Since $\dX$ is not DCI, we know from \cref{Babai} that there is a subgroup~$M$ of $\Aut \dX$, such that $M$ is isomorphic to~$G$, and acts sharply transitively on the vertex set $V(\dX) = G$, but 
	\begin{align} \label{MainThmPf-notRegRep}
	\text{no element of $\Aut \dX$ conjugates $M$ to $G_R$.} 
	\end{align}
It is easy to see that $\Aut \dX \times A_R \times B_R$ is contained in $\Aut \wX$, so 
	\begin{itemize}
	\item $M \times A_R \times B_R \subseteq \Aut \wX$,
	and it is clear that 
	\item this subgroup is isomorphic to~$\wG$ and acts sharply transitively on $G \times A \times B = \wG = V(\wX)$.
	\end{itemize}

If $\wX$ is CI, then \cref{Babai} tells us that some $\varphi \in \Aut \wX$ conjugates $M \times A_R \times B_R$ to the right regular representation of~$\wG$.
We may assume $\varphi(1_{\wG}) = 1_{\wG}$ (by composing $\varphi$ with a translation).
Also, there is no loss of generality in assuming that $|S| \le \bigl( |G| + 1 \bigr)/2$, because we can replace $S$ with its (almost) complement $(G \smallsetminus S) \cup \{1_G\}$. (Recall that \cref{1inS} requires us to keep $1_G$ in~$S$.)
Then the hypotheses of \cref{phiIsGood} are satisfied, so we conclude that the hypotheses of \cref{IfFixG} are satisfied. Hence, the restriction of~$\varphi$ to~$G$ is either 
	an automorphism of~$\dX$ or an isomorphism from $\dX$ to $\dX^-$.

Since $\varphi$ conjugates $M \times A_R \times B_R$ to $\wG_R = G_R \times A_R \times B_R$, and we now know that $\varphi(G) = G$, we can conclude that 
	\[ \text{the restriction $\varphi|_G$ conjugates $M$ to~$G_R$.} \]
This contradicts~\pref{MainThmPf-notRegRep} if $\varphi|_G$ is an automorphism of~$\dX$.

Therefore, $\varphi|_G$ must be an isomorphism from $\dX$ to $\dX^-$. Then $\dX \cong \dX^-$, so, by assumption, there is an automorphism~$\alpha$ of~$G$, such that $\alpha(S) = S^{-1}$, so $\alpha$ is an isomorphism from $\dX$ to $\dX^-$. Also, $\alpha$ normalizes~$G_R$ (since $\alpha$ is a group automorphism). Then the composition $\alpha \circ \varphi|_G$ is an automorphism of~$\dX$ that conjugates $M$ to~$G_R$. This is again a contradiction to~\pref{MainThmPf-notRegRep}.
\end{proof}

\end{document}